\documentclass[reqno]{amsart}

\usepackage[english]{babel}
\usepackage[centering]{geometry}

\usepackage{amsmath,amssymb,amsfonts,amsthm,amscd}
\usepackage{mathrsfs}
\usepackage{amsbsy,bm}
\usepackage{enumerate}
\usepackage{textcomp}
\usepackage{xcolor}
\usepackage{caption}
\usepackage{subcaption}
\usepackage{booktabs}
\usepackage{graphicx}
\usepackage{needspace}

\usepackage{hyperref}
\usepackage{amsmath}
\allowdisplaybreaks[4]
\hypersetup{hidelinks}

\numberwithin{equation}{section}

\makeatletter
\def\subsection{\@startsection{subsection}{2}%
  \z@{.5\linespacing\@plus.7\linespacing}{.3\linespacing}%
  {\normalfont\bfseries}}
\makeatother

\renewcommand{\epsilon}{\varepsilon}

\newtheorem{theorem}{Theorem}[section]
\newtheorem*{theorem*}{Theorem}
\newtheorem{lemma}{Lemma}[section]

\theoremstyle{definition}

\newtheorem{remark}{Remark}[section]

\title[Weighted Anisotropic Curvature Flows]
{Anisotropic-Length-Preserving Weighted Anisotropic Curvature Flows}
\author{Zhishuai Liu, Guoxin Wei}
\address{School of Mathematical Sciences, South China Normal University,
	Guangzhou 510631, People's Republic of China}
\date{}

\email{liuzs@m.scnu.edu.cn}
\email{weiguoxin@tsinghua.org.cn}

\makeatletter
\@namedef{subjclassname@2020}{\textup{2020} Mathematics Subject Classification}
\makeatother



\begin{document}

\begin{abstract}
In this paper, we study the anisotropic-length-preserving weighted anisotropic
curvature flow for smooth convex closed plane curves.
For any smooth, embedded, closed, convex initial curve, the flow exists smoothly for all time, and the anisotropic curvature converges smoothly to that of the boundary of the corresponding homothetic Wulff shape as time tends to infinity.
\end{abstract}

\maketitle

\section{Introduction}

\subsection{Background}
In the past few decades, the plane curve flows have received a lot of attention.
A plane curve flow is a one-parameter family of closed plane curves
$\gamma(\cdot,t)$, defined for $t\in[0,T)$ for some $T>0$.
The curves evolve in the direction of their inward unit normal vector field
$\mathbf{n}$ with normal speed $V$, namely,
\begin{equation}\label{1.1}
  \left\{
  \begin{array}{ll}
  \dfrac{\partial \gamma}{\partial t}(u,t)
  =
  V(x,\kappa,\theta)\mathbf{n},
  & \text{in } S^1\times(0,T),\\[2mm]
  \gamma(u,0)
  =
  \gamma_0(u),
  & \text{on } S^1,
  \end{array}
  \right.
\end{equation}
where $\gamma_0:S^1\to\mathbb{R}^2$ is a given smooth closed curve and
$u\in S^1$ is a parameter independent of time.
The normal speed $V$ is assumed to depend on the curvature $\kappa$, the tangent angle $\theta$ and the position vector $x=\gamma(u,t)$.
Models of curve flows arise in various applications, such as image processing~\cite{25}, phase transitions~\cite{16} and crystal growth~\cite{14}.
For the general flow \eqref{1.1}, the Euclidean length $L(t)=\int_{\gamma} ds$ of the evolving curve $\gamma(t)$ satisfies
\begin{equation*}
  \frac{dL}{dt}=-\int_{\gamma}\kappa Vds,
\end{equation*}
see, for instance, Section~1.3 of Chou--Zhu~\cite{6}, where $ds$ denotes the arc-length element.

Thus, the well-known curve shortening flow $\frac{\partial\gamma}{\partial t}
=\kappa\mathbf{n}$, introduced by Mullins~\cite{23} to describe the motion of grain boundaries, can be regarded as the gradient flow of the length functional, since
$\frac{dL}{dt}= -\int_{\gamma}\kappa^2ds$.
Gage and Hamilton~\cite{7} showed that any embedded closed convex initial curve shrinks to a point in finite time and becomes asymptotically circular after suitable rescaling.
Later, Grayson~\cite{15} extended this result to arbitrary embedded closed initial curves.

Due to the anisotropy of crystalline materials, Mullins' theory was generalized by Gurtin~\cite{17,18}, Angenent and Gurtin~\cite{3,4} and Taylor--Cahn~\cite{29}; see also the monograph of Gurtin~\cite{16}.
A basic anisotropic curvature flow is given by $V=\Psi(\theta)\kappa$, where $\theta$ is the tangent angle of the evolving curve and $\Psi$ is a positive, smooth, $2\pi$-periodic function.
Gage--Li~\cite{10} and Chou--Zhu~\cite{6} investigated this flow and showed that the evolving curve shrinks to a point in finite time and is asymptotic to a self-similar solution.
As a natural generalization, Andrews~\cite{1} considered the power-type curvature flow  $V=\Psi(\theta)\kappa^{\alpha}$ and classified its behavior according to the value of $\alpha$.

More generally, we consider the weighted nonlocal anisotropic curvature flow
\begin{equation}\label{1.2}
\left\{
\begin{array}{l}
\displaystyle
\frac{\partial\gamma}{\partial t}
=\operatorname{sgn}(\alpha)m(\theta)
\left(\kappa_{\sigma}^{\alpha}-\lambda(t)\right)\mathbf{n},
\\
\gamma(\cdot,0)=\gamma_0.
\end{array}
\right.
\end{equation}
Here $m=m(\theta)>0$ is a smooth $2\pi$-periodic function satisfying
$m+m_{\theta\theta}>0$ and $\lambda(t)$ is the nonlocal term defined by
\begin{equation*}
  \lambda(t)
  =\frac{\int_{\gamma}m\kappa_{\sigma}^{\alpha+1}ds}{\int_{\gamma}m\kappa_{\sigma}ds}.
\end{equation*}

When $\alpha>0$ and $m\equiv1$ in \eqref{1.2}, the weighted flow reduces to the
anisotropic-length-preserving curvature equation studied by Liu--Tsai--Wang~\cite{20}, based on the work of \v{S}ev\v{c}ovi\v{c}--Yazaki~\cite{28}:
\begin{equation}\label{1.3}
  \left\{
  \begin{array}{l}
  \displaystyle
  \dfrac{\partial \gamma}{\partial t}
  =\left(\kappa_{\sigma}^{\alpha}
  -\dfrac{\int_{\gamma}\kappa_{\sigma}^{\alpha+1}ds}{\int_{\gamma}\kappa_{\sigma}ds}
  \right)\mathbf{n},
  \\
  \gamma(\cdot,0)=\gamma_0.
  \end{array}
  \right.
\end{equation}
They proved that \eqref{1.3} admits a global smooth solution and that the evolving curves converge smoothly to the boundary of a homothetic Wulff shape as $t\to\infty$.

Related area-preserving and length-preserving curvature flows have been studied by Gage~\cite{8}, Jiang--Pan~\cite{19}, Ma--Cheng~\cite{21}, Ma--Zhu~\cite{22}, Tsai--Wang~\cite{30} and Wang~\cite{32}.
For their anisotropic counterparts in the planar setting, we refer to
Benes--Yazaki--Kimura~\cite{5}, Gao--Zhang~\cite{11}, Liu--Tsai--Wang~\cite{20} and
Pan--Yang~\cite{24}.
In higher dimensions, volume-preserving and locally constrained anisotropic curvature flows have been studied by Wei--Xiong~\cite{33,34}.

More recently, Andrews--Lei--Wei--Xiong~\cite{2} studied volume-preserving flows of smooth, closed, strictly convex hypersurfaces driven by positive powers of the $k$-th anisotropic mean curvature.
They established global existence and convergence to the Wulff shape for all $\alpha>0$.
In particular, their results cover the area-preserving anisotropic curvature flow in the plane, i.e.,
\begin{equation*}
\frac{\partial\gamma}{\partial t}
=\sigma(\theta)\left(\kappa_{\sigma}^{\alpha}
-\frac{\int_{\gamma}\sigma\kappa_{\sigma}^{\alpha}ds}{\int_{\gamma}\sigma ds}
\right)\mathbf{n}.
\end{equation*}
For this flow, they proved exponential convergence of the evolving curves to the Wulff shape.

In this paper, we prove that the flow \eqref{1.7} admits a global smooth solution for any $\alpha>0$.
This result extends the anisotropic length-preserving flow studied by Sun~\cite{26}, which corresponds to the special case $\alpha=1$, $m=\sigma(\theta)$ and the additional symmetry condition $\sigma(\theta+\pi)=\sigma(\theta)$.
Under these assumptions, Sun established the global existence and smooth convergence of the flow.

\subsection{Definitions and notation}

Let $\gamma$ be a smooth, embedded, closed, convex plane curve parametrized by arc length $s$, where embeddedness means that $\gamma$ has no self-intersections.
Denote its unit tangent vector by  $\mathbf{t}=\gamma_s$, and orient $\gamma$ so that the inward unit normal is given by $\mathbf{n}=R_{\frac{\pi}{2}}\mathbf{t}$, where $R_{\frac{\pi}{2}}$ denotes counterclockwise rotation through an angle of $\frac{\pi}{2}$ in $\mathbb{R}^2$.
With the above choice of orientation, we write $\mathbf{t}=-(\cos\theta,\sin\theta),
\mathbf{n}=(\sin\theta,-\cos\theta).$
With this orientation, the Frenet formula takes the form $\gamma_{ss}=\kappa\mathbf{n}$,
and hence  $\mathbf{n}=\frac{\gamma_{ss}}{\|\gamma_{ss}\|}, \kappa=\|\gamma_{ss}\|$.
Throughout, $\gamma$ is said to be convex if $\kappa>0$ everywhere.

When $\gamma$ is convex, it can be parametrized by its tangent angle
$\theta\in S^1:=\mathbb R/2\pi\mathbb Z$.
The relation among the arc-length parameter $s$, the tangent angle $\theta$ and the curvature $\kappa$ is given by
\begin{equation*}
  \kappa\,ds=d\theta.
\end{equation*}
 Define the support function of $\gamma$ with respect to
the origin $O$ by
\begin{equation*}
  h(\theta)=-\langle\gamma,\mathbf{n}\rangle.
\end{equation*}
Thus, $h$ is the scalar projection of the position vector of $\gamma$
onto the outer unit normal vector $-\mathbf{n}$. The curvature and the
support function are related by
\begin{equation*}
  \kappa(\theta)=\frac{1}{h(\theta)+h''(\theta)}.
\end{equation*}

By Proposition~2.1 of Chou--Zhu~\cite{6}, any smooth $2\pi$-periodic
function $h(\theta)$ satisfying  $h''+h>0$
determines a smooth, embedded, closed, convex plane curve by
$\gamma(\theta)= \bigl( -h(\theta)\sin\theta-h'(\theta)\cos\theta,\,
  h(\theta)\cos\theta-h'(\theta)\sin\theta\bigr).$
Moreover, two curves determined by $h$ and $\widetilde{h}$ differ by a
translation if and only if $h-\widetilde{h}=C_1\cos\theta+C_2\sin\theta$
for some constants $C_1$ and $C_2$.

The Euclidean length $L$ of $\gamma$ is given by
\begin{equation*}
  L =\int_{\gamma}ds
  =\int_0^{2\pi}\frac{1}{\kappa}\,d\theta
  = \int_0^{2\pi}h\,d\theta,
\end{equation*}
while the enclosed area $A$ is given by
\begin{equation*}
  A =\frac{1}{2}\int_{\gamma}\langle\gamma,-\mathbf{n}\rangle\,ds
  =\frac{1}{2}\int_{\gamma}h\,ds
  =\frac{1}{2}\int_0^{2\pi}\left(h^2-(h')^2\right)\,d\theta.
\end{equation*}

Let $\sigma$ be a given positive smooth function satisfying
\begin{equation}\label{1.4}
   \sigma(\theta+2\pi)=\sigma(\theta),
   \qquad
  \sigma''(\theta)+\sigma(\theta)>0,
  \qquad
  \theta\in S^1.
\end{equation}
Set
$\varphi(\theta)=\sigma''(\theta)+\sigma(\theta)$.
By \eqref{1.4}, $\varphi>0$ on $S^1$. We denote
\begin{equation*}
  \Lambda_0=\min_{\theta\in S^1}\varphi(\theta),
  \qquad
  \Lambda_1=\max_{\theta\in S^1}\varphi(\theta),
\end{equation*}
so that
$0<\Lambda_0\leq\varphi(\theta)\leq\Lambda_1$ for all $\theta\in S^1$.

The anisotropic curvature of $\gamma$ associated with $\sigma(\theta)$ is defined by
\begin{equation}\label{1.5}
  \kappa_{\sigma}(\theta)=\varphi(\theta)\kappa(\theta)
  =\bigl(\sigma''(\theta)+\sigma(\theta)\bigr)\kappa(\theta),
\end{equation}
and the anisotropic length is defined by
\begin{equation*}
  L_{\sigma}=\int_{\gamma}\sigma ds=\int_0^{2\pi}\frac{\sigma}{\kappa} d\theta.
\end{equation*}
Moreover,
\begin{equation*}
  \int_{\gamma}\kappa_{\sigma} ds
  =\int_{\gamma}(\sigma+\sigma'')\kappa ds
  =\int_0^{2\pi}(\sigma+\sigma'') d\theta
  =\int_0^{2\pi}\sigma  d\theta,
\end{equation*}
which is independent of $\gamma$.

The definition of the Wulff shape is taken from~\cite{28}. We define the
Wulff shape as the intersection of the half-spaces containing the origin
$O$, which are determined by a family of hyperplanes:
\begin{equation*}
  W_\sigma
  =\bigcap_{\theta\in S^1}
  \left\{\mathbf{x}\in\mathbb{R}^2:-\mathbf{x}\cdot\mathbf{n}(\theta)
  \leq\sigma(\theta)\right\},
\end{equation*}
where $\sigma(\theta)$ is a given function satisfying \eqref{1.4}. If the
boundary of the Wulff shape $W_\sigma$ is smooth, it can be parametrized
as follows:
\begin{equation*}
  \partial W_\sigma
  =\left\{\mathbf{x}:\mathbf{x}
  =-\sigma(\theta)\mathbf{n}+\sigma'(\theta)\mathbf{t},
  \quad
  \theta\in S^1
  \right\},
\end{equation*}
and its support function is $\sigma(\theta)$. The curvature of
$\partial W_\sigma$ is given by
$\kappa=(\sigma(\theta)+\sigma''(\theta))^{-1}
=\varphi(\theta)^{-1}$,
so the anisotropic curvature $\kappa_\sigma$ of $\partial W_\sigma$
is equal to the constant $1$. Moreover, the area $|W_\sigma|$ of the
Wulff shape $W_\sigma$ satisfies
\begin{equation}\label{1.6}
  |W_\sigma|
  =-\frac{1}{2}\int_{\partial W_{\sigma}}\mathbf{x}\cdot\mathbf{n}ds
  =\frac{1}{2}\int_{\partial W_{\sigma}} \sigma ds
  =\frac{1}{2}\int_0^{2\pi}\sigma\varphi  d\theta
  =\frac{1}{2}L_{\sigma}(\partial W_{\sigma}).
\end{equation}
In particular, $|W_1|=\pi$ if $\sigma\equiv1$.

\subsection{Main theorem}

In this paper, we study the power-type nonlocal anisotropic curvature flow \eqref{1.2} of smooth, convex, closed plane curves with $\alpha>0$ under an anisotropic-length-preserving (ALP) constraint, namely,
\begin{equation}\label{1.7}
  \left\{
  \begin{array}{l}
  \displaystyle
  \dfrac{\partial \gamma}{\partial t}
  =m\left(\kappa_{\sigma}^{\alpha}-\dfrac{\int_{\gamma}m\kappa_{\sigma}^{\alpha+1}ds}
  {\int_{\gamma}m\kappa_{\sigma}ds}\right)\mathbf{n},
  \\[2ex]
  \gamma(\cdot,0)=\gamma_0.
  \end{array}
  \right.
\end{equation}
Here $\alpha>0$, and $m=m(\theta)$ is a smooth positive $2\pi$-periodic function satisfying $m+m_{\theta\theta}>0$ on $S^1$.
Moreover, $\gamma=\gamma(u,t)$, where $u$ is independent of time and $\mathbf{n}$ denotes the inward unit normal vector.

Our main theorem is the following.

\begin{theorem}
Let $\gamma_0$ be a smooth, embedded, closed, convex plane curve.
Then the flow \eqref{1.7} admits a global smooth solution, and the
evolving curves remain convex for all
$t\geq0$. Along the flow \eqref{1.7}, the anisotropic length is
preserved and the enclosed area is nondecreasing. Moreover, as
$t\to\infty$, the anisotropic curvature
$\kappa_\sigma$ of the evolving curves converges smoothly to the
anisotropic curvature of the boundary of the corresponding homothetic
Wulff shape determined by $\sigma$.
\end{theorem}

We organize the following sections in this way.  In Section~2, we
derive the evolution equations for the relevant geometric quantities,
establish the preservation of the anisotropic length and the
monotonicity of the enclosed area, and prove the preservation of
convexity and the local existence of solutions to \eqref{1.7}. In Section~3,
we derive estimates for the support function and the anisotropic
curvature by Tso's method, obtain derivative estimates by the Bernstein
technique, and establish the global existence of the solutions. In
Section~4, we study the long-time behavior of the anisotropic curvature
and prove its smooth convergence to the corresponding constant limit.
\\

\noindent {\bf Acknowledgments}. We would like to thank Yong Wei for his valuable comments and suggestions.

\section{Preparation}

A  convex curve $\gamma$ can be parametrized by its tangent angle
$\theta$. Thus, for a family of  convex curves $\gamma(u,t)$,
the tangent angle may be used as a parameter for $\gamma(\cdot,t)$ at
each fixed time $t$. In general, however, $\theta$ depends on both $u$
and $t$. By introducing a suitable tangential velocity $G\mathbf{t}$,
we may reparametrize the evolving curves so that $\theta$ is independent
of time. Accordingly, we consider the reparametrized flow
\begin{equation}\label{2.1}
  \left\{
  \begin{array}{ll}
  \displaystyle
  \dfrac{\partial\gamma}{\partial t}(u,t)=V\mathbf{n}+G\mathbf{t},
  & \text{in } S^1\times(0,T),\\[2mm]
  \gamma(u,0)=\gamma_0(u),
  & \text{on } S^1.
  \end{array}
  \right.
\end{equation}
Here $V=m(\theta)(\kappa_{\sigma}^{\alpha}-\lambda(t))$, where the nonlocal term $\lambda(t)$ is specified below.
The tangential velocity $G$ will be chosen so that the tangent angle $\theta$ remains independent of time.
As is well known (see, for example,~\cite{6,8}), the addition of a tangential
velocity does not affect the geometric evolution of the curve.
Indeed, the original flow and the reparametrized flow differ only by a time-dependent reparametrization.

Let $g=\|\frac{\partial\gamma}{\partial u}\|$.
Following~\cite{6}, we have
\begin{equation}\label{2.2}
  \begin{aligned}
  \frac{\partial g}{\partial t}
  &=\left(-\kappa V+G_s\right)g,
  \\
  \frac{\partial\theta}{\partial t}
  &=V_s+\kappa G,
  \\
  \frac{\partial\kappa}{\partial t}
  &=\left(V\kappa-G_s\right)\kappa+\left(V_s+G\kappa\right)_s.
  \end{aligned}
\end{equation}
Choosing $G=-\frac{V_s}{\kappa}$, we obtain $\frac{\partial\theta}{\partial t}=0$.
Hence, $\theta$ may be regarded as a time-independent parameter along the reparametrized flow.

For the flow \eqref{1.7}, the nonlocal term is given by
\begin{equation}\label{2.3}
  \lambda(t)
  =\frac{\int_{\gamma}m\kappa_{\sigma}^{\alpha+1}ds}{\int_{\gamma}m\kappa_{\sigma}ds}
  =\frac{\int_0^{2\pi}m\varphi\kappa_{\sigma}^{\alpha}d\theta}{\int_0^{2\pi}m\varphi d\theta}.
\end{equation}

Using $\theta$ as a time-independent parameter, the reparametrized flow
\eqref{2.1} is equivalent to the following evolution problem for the
curvature $\kappa(\theta,t)$:
\begin{equation}\label{2.4}
  \left\{
  \begin{array}{ll}
  \displaystyle
  \frac{\partial\kappa}{\partial t}
  =\kappa^2(m(\kappa_{\sigma}^{\alpha})_{\theta\theta}
  +2m_{\theta}(\kappa_{\sigma}^{\alpha})_{\theta}
  +(m+m_{\theta\theta})(\kappa_{\sigma}^{\alpha}-\lambda(t))),
  & (\theta,t)\in S^1\times(0,T),
  \\[2mm]
  \kappa(\theta,0)=\kappa_0(\theta),
  & \theta\in S^1,
  \end{array}
  \right.
\end{equation}
where $\kappa_0(\theta)>0$ denotes the curvature of the initial curve
$\gamma_0$.

Set
\begin{equation*}
  \omega=\kappa_{\sigma}^{\alpha},
  \qquad
  \omega_0=(\varphi\kappa_0)^\alpha.
\end{equation*}
Recalling that $\varphi(\theta)=\sigma(\theta)+\sigma''(\theta)$,
problem \eqref{2.4} can be rewritten equivalently as
\begin{equation}\label{2.5}
  \left\{
  \begin{array}{ll}
  \displaystyle
  \frac{\partial\omega}{\partial t}
  =\frac{\alpha}{\varphi}\omega^{1+\frac{1}{\alpha}}\left(m\omega_{\theta\theta}
  +2m_{\theta}\omega_{\theta}+(m+m_{\theta\theta})(\omega-\lambda(t))\right),
  & (\theta,t)\in S^1\times(0,T),
  \\[2mm]
  \omega(\theta,0)=\omega_0(\theta),
  & \theta\in S^1.
  \end{array}
  \right.
\end{equation}

Before studying the monotonicity of the relevant geometric quantities, we recall several geometric inequalities that will be used below.
For an embedded closed plane curve, the classical isoperimetric inequality states that its Euclidean length $L$ and enclosed area $A$ satisfy
\begin{equation*}
  L^2\geq4\pi A.
\end{equation*}

Its anisotropic counterpart is the anisotropic isoperimetric inequality \cite[Theorem~2.8]{9}.
More precisely, let $\sigma$ be a positive smooth $2\pi$-periodic function satisfying
$\sigma''+\sigma>0$.
Then, for any embedded closed convex curve $\gamma$, its anisotropic length $L_{\sigma}$ and enclosed area $A$ satisfy
\begin{equation}\label{2.6}
  L_{\sigma}^2 \geq 4|W_{\sigma}|A,
\end{equation}
where $|W_{\sigma}|$ denotes the area of the Wulff shape $W_\sigma$
associated with $\sigma$.

Finally, we define the anisoperimetric ratio of $\gamma(t)$ by
\begin{equation}\label{2.7}
  \Pi_{\sigma}(t)=\frac{L_{\sigma}^2(t)}{4|W_{\sigma}|A(t)}.
\end{equation}
By the anisotropic isoperimetric inequality \eqref{2.6}, we have
$\Pi_{\sigma}(t)\geq1$.

We shall also need the following consequence of H\"older's inequality.

\begin{lemma}\label{1}
Let $\gamma$ be a smooth, embedded, closed, convex plane curve, and let $m>0$ be a smooth $2\pi$-periodic function satisfying $m''+m>0$.
Then for any $\alpha>0$, we have
\begin{equation}\label{2.8}
  \int_{\gamma}m\kappa_{\sigma}ds\int_{\gamma}m\kappa_{\sigma}^{\alpha}ds
  \leq
  \int_{\gamma}mds\int_{\gamma}m\kappa_{\sigma}^{\alpha+1}ds.
\end{equation}
\end{lemma}

\begin{proof}
By H\"older's inequality, we have
\begin{equation*}
  \int_{\gamma}m\kappa_{\sigma}ds
  =\int_{\gamma}m^{\frac{\alpha}{\alpha+1}}(m\kappa_{\sigma}^{\alpha+1})^{\frac{1}{\alpha+1}}ds
  \leq
  \left(\int_{\gamma}mds\right)^{\frac{\alpha}{\alpha+1}}
  \left(\int_{\gamma}m\kappa_{\sigma}^{\alpha+1}ds\right)^{\frac{1}{\alpha+1}},
\end{equation*}
and
\begin{equation*}
  \int_{\gamma}m\kappa_{\sigma}^{\alpha}ds
  =\int_{\gamma}m^{\frac{1}{\alpha+1}}(m\kappa_{\sigma}^{\alpha+1})^{\frac{\alpha}{\alpha+1}}ds
  \leq
  \left(\int_{\gamma}mds\right)^{\frac{1}{\alpha+1}}
  \left(\int_{\gamma}m\kappa_{\sigma}^{\alpha+1}ds\right)^{\frac{\alpha}{\alpha+1}}.
\end{equation*}
Multiplying the above two inequalities together yields \eqref{2.8}.
\end{proof}

\begin{lemma}\label{2}
For the ALP flow \eqref{1.7}, the anisotropic length $L_{\sigma}(t)$ is preserved while the enclosed area $A(t)$ is nondecreasing in time.
\end{lemma}

\begin{proof}
By the first variation formula for the anisotropic length, we have
\begin{equation}\label{2.9}
  \begin{aligned}
  \frac{dL_\sigma(t)}{dt}
  &=-\int_{\gamma}\kappa_{\sigma} Vds
  \\
  &=
  -\int_{\gamma}m\kappa_{\sigma}^{\alpha+1}ds
  +\frac{\int_{\gamma}m\kappa_{\sigma}^{\alpha+1}ds}
  {\int_{\gamma}m\kappa_{\sigma}ds}\int_{\gamma}m\kappa_{\sigma}ds
  \\
  &=0.
  \end{aligned}
\end{equation}
According to the formula in Sect. 1.3 of \cite{6}, we have
\begin{equation}\label{2.10}
  \begin{aligned}
  \frac{dA(t)}{dt}
  &=-\int_{\gamma}Vds
  \\
  &=
  -\int_{\gamma}m\kappa_{\sigma}^{\alpha}ds
  +\frac{\int_{\gamma}m\kappa_{\sigma}^{\alpha+1}ds}
  {\int_{\gamma}m\kappa_{\sigma}ds}\int_{\gamma}mds
  \\
  &\geq0,
  \end{aligned}
\end{equation}
where the last inequality follows from \eqref{2.8}. Hence, the
anisotropic length is preserved and the enclosed area is nondecreasing.
\end{proof}

\begin{lemma}\label{3}
A strictly convex smooth plane curve $\gamma(0)$ which evolves according to \eqref{1.7} remains convex as long as the flow exists.
\end{lemma}

\begin{proof}
Suppose that the flow exists on the time interval $[0,T)$, $T\leq\infty$.
By \eqref{2.5}, the equation for $\omega=\kappa_\sigma^\alpha$ can be written as
\begin{equation*}
  \omega_t
=a(\theta,t)\omega_{\theta\theta}+b(\theta,t)\omega_{\theta}+c(\theta,t)\omega,
\qquad
(\theta,t)\in S^1\times(0,T),
\end{equation*}
where $a(\theta,t)=\frac{\alpha m}{\varphi}\omega^{1+\frac{1}{\alpha}},  b(\theta,t)=\frac{2\alpha m_\theta}{\varphi}\omega^{1+\frac{1}{\alpha}}$  and
$c(\theta,t)=\frac{\alpha(m+m_{\theta\theta})}{\varphi}
\omega^{\frac{1}{\alpha}}\bigl(\omega-\lambda(t)\bigr)$.

Fix $t_1<T$. There exists a constant $C(t_1)>0$ such that $\omega(\theta,t)\leq C(t_1)$ for $(\theta,t)\in S^1\times[0,t_1]$.
By \eqref{2.3}, we have $0<\lambda(t)\leq C(t_1)$ for all $t\in[0,t_1]$.
Since $c(\theta,t)=\frac{\alpha(m+m_{\theta\theta})}{\varphi}
\omega^{\frac{1}{\alpha}}\bigl(\omega-\lambda(t)\bigr)$, there exists a constant $B=B(t_1)>0$ such that $|c(\theta,t)|\leq B$ for $(\theta,t)\in S^1\times[0,t_1]$.

Set $\widehat{\omega}=e^{(B+1)t}\omega$.
Then $\widehat{\omega}$ satisfies
\begin{equation*}
  \widehat{\omega}_t
  =a(\theta,t)\widehat{\omega}_{\theta\theta}
  +b(\theta,t)\widehat{\omega}_{\theta}
  +\bigl(c(\theta,t)+B+1\bigr)\widehat{\omega}.
\end{equation*}
Since $|c(\theta,t)|\leq B$, we have $c(\theta,t)+B+1\geq1$.
Hence, by the maximum principle,
\begin{equation*}
  \min_{S^1}\widehat{\omega}(\theta,t)
  \geq
  \min_{S^1}\widehat{\omega}(\theta,0),
  \quad
  0<t\leq t_1.
\end{equation*}
Since the initial curve $\gamma_0$ is convex, i.e., $\omega(\theta,0)>0$, it follows that
\begin{equation*}
  \omega(\theta,t)>0,
  \quad
  (\theta,t)\in S^1\times[0,t_1].
\end{equation*}
Since $t_1<T$ is arbitrary, we have
\begin{equation*}
  \omega(\theta,t)>0,
  \quad
  (\theta,t)\in S^1\times[0,T).
\end{equation*}
\end{proof}

\begin{lemma}\label{4}
There exists $T_0>0$ such that there is a unique smooth solution to the flow \eqref{1.7} on $[0,T_0)$. Moreover, the  solution exists for all time if the curvature $\kappa$ (or equivalently, the anisotropic curvature $\kappa_{\sigma}$) of the evolving curves $\gamma(\cdot,t)$ does not blow up at finite time.
\end{lemma}

\begin{proof}
The unique existence of the local solution to flow \eqref{1.7} can be proved by mimicking the proof in~\cite{24}, where the classical Leray-Schauder fixed point theory is employed.
The curvature of the evolving curves does not blow up in finite time means that for any given time $T>0$ the curvature $\kappa$ (or equivalently the quantity $\omega$) has an upper bound $C(T)$, where $C(T)$ is a positive constant depending on time $T$, i.e.
\begin{equation*}
0<\kappa\leq C(T),
\quad(\theta,t)\in S^1\times[0,T).
\end{equation*}
Combining with the lower bound obtained in the proof of Lemma~\ref{3}, we can use Bernstein techniques  to establish the bound for all the derivatives of $\kappa$, that is
\begin{equation*}
|\partial_{\theta}^{i}\kappa(\theta,t)|\leq C_i(T),
\quad (\theta,t)\in S^1 \times [0,T),
\quad i=1,2,3,\cdots
\end{equation*}
for some positive constants $C_i(T), i=1,2,3,\ldots$.
By local existence, we can assume that the solution exists on a maximal time interval $[0,T_{\max})$.
If $T_{\max}<\infty$, we can use the above estimates for $\partial_{\theta}^{i}\kappa(\theta,t)(i=0,1,2,3\ldots)$ to extend the solution beyond $T_{\max}$, which is a contradiction to the definition of $T_{\max}$.
\end{proof}


\section{Estimate of the curvature}

In this section, we use Tso's method~\cite{31} to estimate the bound
of curvature, and use the Bernstein technique to estimate the
derivatives of curvature.


\subsection{Estimate of the support function}

Let $h(\theta,t)$ be the support function of the evolving curves
$\gamma(\theta,t)$ under the flow \eqref{1.7},
with respect to the origin $O$ as defined in Section~1.2:
\begin{equation*}
  h(\theta,t)=-\langle\gamma(\theta,t),\mathbf {n}\rangle,
\quad
(\theta,t)\in S^1\times[0,T),
\end{equation*}
where $\theta$ is the tangent angle of $\gamma(\theta,t)$.
More generally, given a point $P=(a,b)$ in the plane, the support function $s(\theta,t)$ with respect to $P$ is defined by
\begin{equation*}
  s(\theta,t)=-\langle\gamma(\theta,t)-P,\mathbf {n}\rangle
  =h(\theta,t)+a\sin\theta-b\cos\theta,
\quad
(\theta,t)\in S^1\times[0,T).
\end{equation*}
A direct computation shows that the evolution equation for the support
function is independent of the choice of reference point.

To apply Tso's method, we first derive a two-side estimate for the support function of the evolving curves.
Let $\gamma(t)$ be an embedded, closed, convex plane curve. By Bonnesen's inequality, one has
\begin{equation*}
  rL(t)-A(t)-\pi r^2\geq0,
\end{equation*}
for all $r_{\mathrm{in}}(t)\leq r\leq r_{\mathrm{out}}(t)$, where
$r_{\mathrm{in}}(t)$ and $r_{\mathrm{out}}(t)$ denote the inradius and
outradius of the convex domain enclosed by $\gamma(t)$, respectively.
Thus, we have
\begin{equation}\label{3.1}
0<
\frac{L(t)-\sqrt{L^2(t)-4\pi A(t)}}{2\pi}
\leq
r_{\mathrm{in}}(t)
\leq
r_{\mathrm{out}}(t)
\leq
\frac{L(t)+\sqrt{L^2(t)-4\pi A(t)}}{2\pi}.
\end{equation}
Since the anisotropic length $L_{\sigma}(t)$ is preserved in time, we have
\begin{equation}\label{3.2}
L(t)+\sqrt{L^2(t)-4\pi A(t)}\leq2L(t)\leq\Lambda_2\int_{\gamma}\sigma ds
=\Lambda_2L_{\sigma}(t)\leq\Lambda_2L_{\sigma}(0),
\end{equation}
where  $\Lambda_2=\max_{\theta\in S^1}\frac{2}{\sigma(\theta)}.$
It follows that
\begin{equation}\label{3.3}
L(t)-\sqrt{L^2(t)-4\pi A(t)}
=\frac{4\pi A(t)}{L(t)+\sqrt{L^2(t)-4\pi A(t)}}
\geq
\frac{4\pi A(0)}{\Lambda_2L_{\sigma}(0)}.
\end{equation}
Combining \eqref{3.1}, \eqref{3.2} and \eqref{3.3}, we obtain
\begin{equation*}
0<2a\leq r_{\mathrm{in}}(t)\leq r_{\mathrm{out}}(t)\leq2b,
\end{equation*}
where  $2a=\frac{2A(0)}{\Lambda_2L_\sigma(0)},
2b=\frac{\Lambda_2L_\sigma(0)}{2\pi}$.
For a fixed time $t_0\in[0,T)$,  let $P_{t_0}\in\mathbb R^2$ be the center of some maximal inscribed circle of $\gamma(\cdot,t_0)$ and  $E(t_0)$ be a circle enclosed by $\gamma(\cdot,t_0)$  with radius $r(t_0)=2a$ centered at $P_{t_0}$.
From time $t_0$, let $E(t)$ evolve according to the flow
\begin{equation}\label{3.4}
\frac{\partial E}{\partial t}=\Lambda_3\kappa^\alpha\mathbf n,
\end{equation}
where $\Lambda_3=\max_{\theta\in S^1}m(\theta)\varphi^\alpha(\theta)$.
Note that each $E(t)$ is a circle as long as it exists, and a direct computation
shows that the radius of $E(t)$ satisfies
\begin{equation*}
  r(t)=\left[r^{\alpha+1}(t_0)-(\alpha+1)\Lambda_3(t-t_0)\right]^{\frac{1}{\alpha+1}},
  \quad
  t\in\left[t_0,t_0+\frac{r^{\alpha+1}(t_0)}{(\alpha+1)\Lambda_3}\right).
\end{equation*}
By comparing the velocity of the flow \eqref{3.4} with the flow \eqref{1.7}, we know that the circle $E(t)$ is enclosed by the curve $\gamma(t)$ for
$t_0\leq t<\min\left\{t_0+\frac{r^{\alpha+1}(t_0)}{(\alpha+1)\Lambda_3},T\right\}$.
Consequently, the support function $s(\theta,t)$ of $\gamma(\cdot,t)$ with respect to $P_{t_0}$ satisfies
\begin{equation*}
  s(\theta,t)\geq r(t),
   \quad
   \theta\in S^1,
   \quad
   t\in[t_0,\min\{t_0+\frac{r^{\alpha+1}(t_0)}{(\alpha+1)\Lambda_3},T\}).
\end{equation*}
Define
\begin{equation*}
  \Delta t
:=\frac{r^{\alpha+1}(t_0)}{2^{\alpha+1}(\alpha+1)\Lambda_3}
=\frac{a^{\alpha+1}}{(\alpha+1)\Lambda_3}.
\end{equation*}
Then the radius $r(t)$ of $E(t)$ satisfies
\begin{equation*}
  r(t_0)\geq r(t)\geq\frac{r(t_0)}{2}=a,
\quad
t\in[t_0,\min\{t_0+\Delta t,T\}).
\end{equation*}
Thus we have
\begin{equation*}
  s(\theta,t)\geq r(t)\geq a,
  \quad
  (\theta,t)\in S^1\times[t_0,\min\{t_0+\Delta t,T\}).
\end{equation*}
For the upper bound, we have
\begin{equation*}
  s(\theta,t)
\leq
\frac{L(t)}{2}
\leq
\frac{\Lambda_2L_\sigma(t)}{4}
\leq
\frac{\Lambda_2L_\sigma(0)}{4},
\quad
(\theta,t)\in S^1\times[t_0,\min\{t_0+\Delta t,T\}).
\end{equation*}
By defining
\begin{equation*}
  2d=a,
\quad
2D=\frac{\Lambda_2L_{\sigma}(0)}{4},
\end{equation*}
we obtain the desired two-side bound for the support function.

\begin{lemma}\label{5}
Let $\gamma(t):S^1\times[0,T)\to\mathbb{R}^2$ be a smooth convex solution to the flow \eqref{1.7}.
For any fixed time $t_0\in[0,T)$, there exists a number $\Delta t>0$ which is independent of the time $t_0$, such that the support function $s(\theta,t)$ with respect to a suitably
chosen point $P_{t_0}$ inside $\gamma(t_0)$ satisfies
\begin{equation}\label{3.5}
2d\leq s(\theta,t)\leq 2D,
\qquad
(\theta,t)\in S^1\times[t_0,\min\{t_0+\Delta t,T\}),
\end{equation}
for some constants $d$ and $D$.
\end{lemma}

\begin{remark}
The constants $d$, $D$, and $\Delta t$ in Lemma~\ref{5} are independent of the initial time $t_0$.
Consequently, for every $\xi\in[0,T)$, one may choose a point $P_\xi$ inside the convex
domain enclosed by $\gamma(\xi)$ such that the support function $s_{\xi}(\theta,t)$ of $\gamma(t)$ with respect to $P_{\xi}$ satisfies
\begin{equation*}
    2d\leq s_{\xi}(\theta,t)\leq 2D,
    \quad
    (\theta,t)\in S^1\times[\xi,\min\{\xi+\Delta t,T\}).
\end{equation*}
Thus, the same two-side estimate is available on every translated time interval of length $\Delta t$.
\end{remark}


\subsection{Upper bound of the curvature}

In this subsection, we use the estimate for the support function to obtain a time-independent upper bound for the anisotropic curvature under the flow \eqref{1.7}.
This estimate will be used to prove the global existence of the solution and the convergence of the anisotropic curvature.

\begin{lemma}\label{6}
Let $\gamma(t):S^1\times[0,T)\rightarrow \mathbb{R}^2$ be a smooth, convex, closed
solution of the flow \eqref{1.7}.
There exists a constant $C_0>0$, independent of time, such that
\begin{equation}\label{3.6}
0<\omega(\theta,t)
=\kappa_\sigma^\alpha(\theta,t)
=\bigl(\varphi(\theta)\kappa(\theta,t)\bigr)^\alpha
\leq C_0,
\quad
(\theta,t)\in S^1\times[0,T).
\end{equation}
\end{lemma}

\begin{proof}
\textit{Case 1: $0<T\leq\Delta t$.}

 Consider the quantity
\begin{equation*}
  Q(\theta,t)=\frac{m\omega(\theta,t)}{s(\theta,t)-d},
\quad
(\theta,t)\in S^1\times[0,T).
\end{equation*}
Under the flow,  the evolution of the support function $s(\theta,t)$ is
\begin{equation*}
\frac{\partial s}{\partial t}=-m\bigl(\omega-\lambda(t)\bigr),
\quad
(\theta,t)\in S^1\times[0,T),
\end{equation*}
where $\lambda(t)$ is given by \eqref{2.3}.
Direct computation shows that the evolution equation of $Q$ on $S^1\times[0,T)$ is
\begin{equation*}
  \frac{\partial Q}{\partial t}
=\frac{\alpha m\omega^{1+\frac{1}{\alpha}}}{\varphi(s-d)}\left((m\omega)_{\theta\theta}
+m\omega-(m+m_{\theta\theta})\lambda(t)\right)
+\frac{m^2\omega(\omega-\lambda(t))}{(s-d)^2}.
\end{equation*}

Note that
\begin{equation*}
  Q_{\theta}=\frac{(m\omega)_{\theta}}{s-d}-\frac{m\omega s_{\theta}}{(s-d)^2},
\end{equation*}
and
\begin{equation*}
  Q_{\theta\theta}
=\frac{(m\omega)_{\theta\theta}}{s-d}-\frac{2(m\omega)_{\theta} s_{\theta}}{(s-d)^2}
-\frac{m\omega s_{\theta\theta}}{(s-d)^2}+\frac{2m\omega s_{\theta}^2}{(s-d)^3}.
\end{equation*}
Therefore  we have
\begin{align*}
Q_t
&=\frac{\alpha m\omega^{1+\frac{1}{\alpha}}}{\varphi}Q_{\theta\theta}
+\frac{2\alpha m\omega^{1+\frac{1}{\alpha}}s_\theta}{\varphi(s-d)}Q_\theta
+\frac{\alpha m^2\omega^{2+\frac{1}{\alpha}}s_{\theta\theta}}{\varphi(s-d)^2}
-\frac{\alpha m(m+m_{\theta\theta})\omega^{1+\frac{1}{\alpha}}}{\varphi(s-d)} \lambda(t)                                                      \\
&\quad
+\frac{\alpha m^2\omega^{2+\frac{1}{\alpha}}}{\varphi(s-d)}
+\frac{m^2\omega(\omega-\lambda(t))}{(s-d)^2}
\\
\quad
&\leq
\frac{\alpha m\omega^{1+\frac{1}{\alpha}}}{\varphi}Q_{\theta\theta}
+\frac{2\alpha m\omega^{1+\frac{1}{\alpha}}s_{\theta}}{\varphi(s-d)}Q_{\theta}
+\frac{(\alpha+1)m^2\omega^2}{(s-d)^2}
-\frac{\alpha dm^2\omega^{2+\frac{1}{\alpha}}}{\varphi(s-d)^2}
\\
\quad
&\leq
\frac{\alpha m\omega^{1+\frac{1}{\alpha}}}{\varphi}Q_{\theta\theta}
+\frac{2\alpha m\omega^{1+\frac{1}{\alpha}}s_{\theta}}{\varphi(s-d)}Q_{\theta}
+Q^2\left(\alpha+1-\frac{\alpha d^{1+\frac{1}{\alpha}}}
{\varphi m^{\frac{1}{\alpha}}}Q^{\frac{1}{\alpha}}\right).
\end{align*}
Fix $t\in(0,T)$ and suppose that  $\max_{S^1\times[0,t]}Q=Q(\theta_0,t_0)$ for some $(\theta_0,t_0)\in S^1\times[0,t]$.
If $t_0>0$,  then at the point $(\theta_0,t_0)$,
\begin{equation*}
  Q_t(\theta_0,t_0)\geq0,
\quad
Q_\theta(\theta_0,t_0)=0,
\quad
Q_{\theta\theta}(\theta_0,t_0)\leq0.
\end{equation*}
This implies that
\begin{equation*}
  Q(\theta_0,t_0)
\leq
\left(1+\frac{1}{\alpha}\right)^{\alpha}
\frac{m(\theta_0)\varphi^{\alpha}(\theta_0)}{d^{\alpha+1}}
\leq
\left(1+\frac1\alpha\right)^{\alpha}\frac{\Lambda_3}{d^{\alpha+1}}.
\end{equation*}
So we have
\begin{equation*}
  \max_{S^1\times[0,t]}Q
\leq
\max\left\{\max_{S^1}Q(\theta,0),\left(1+\frac{1}{\alpha}\right)
^{\alpha}\frac{\Lambda_3}{d^{\alpha+1}}\right\}
:=M,
\end{equation*}
and
\begin{equation*}
  \max_{S^1\times[0,t]}\omega\leq\frac{M(2D-d)}{\min_{S^1}m}:=M_1.
\end{equation*}
Since $t<T$ is arbitrary, we have
\begin{equation*}
  \sup_{S^1\times[0,T)}\omega\leq M_1.
\end{equation*}

\textit{Case 2: $T>\Delta t$.}

Let $\xi\in(0,T-\Delta t)$. Consider the evolution problem on the time interval $[\xi,\xi+\Delta t]$.
As noted in Remark~3.1, we can choose a suitable point $P$ as the origin so that the support function with respect to $P$ of $\gamma(\cdot,t)$ still satisfies
\begin{equation*}
  2d\leq s_{\xi}(\theta,t)\leq 2D,
    \quad
    (\theta,t)\in S^1\times[\xi,\xi+\Delta t].
\end{equation*}
As in Case 1, we consider the evolution of the quantity
$Q(\theta,t):=\frac{m\omega}{s_{\xi}-d}$.
Let
\begin{equation*}
  Q^*=\left(\frac{2(\alpha+1)}{\alpha}\right)^{\alpha}\frac{\Lambda_3}{d^{\alpha+1}}.
\end{equation*}
Then whenever
\begin{equation*}
  Q_{\max}=Q(\theta,t)\geq Q^*,
  \quad
  t \in[\xi,\xi+\Delta t],
\end{equation*}
we have
\begin{equation*}
  \frac{d}{dt}Q_{\max}(t)\leq -(\alpha+1)Q_{\max}^2(t),
  \quad
  t \in(\xi,\xi+\Delta t].
\end{equation*}
The comparison principle therefore yields
\begin{equation*}
  Q_{\max} \leq \max\left\{ Q^*,\frac{1}{(\alpha+1)(t-\xi)}\right\},
  \quad
 t\in(\xi,\xi+\Delta t].
\end{equation*}
In particular, we get
\begin{equation*}
\frac{m\omega}{s_{\xi}-d}\leq \max\left\{Q^*,\frac{1}{(\alpha+1)(t-\xi)}\right\},
\quad
(\theta,t) \in S^1\times(\xi,\xi+\Delta t],
\end{equation*}
and then by the two-side bound on the support function we conclude that
\begin{equation*}
  \omega\leq
  \frac{2D-d}{\min_{S^1}m} \max\left\{Q^*,\frac{1}{(\alpha+1)(t-\xi)}\right\},
  \quad
  t\in(\xi,\xi+\Delta t].
\end{equation*}
In particular, when we take $t=\xi+\Delta t$, we have
\begin{equation*}
  \omega_{\max}(\xi+\Delta t)
  \leq
  \frac{2D-d}{\min_{S^1}m}\max\left\{Q^*,\frac{1}{(\alpha+1)\Delta t}\right\}:=M_2.
\end{equation*}
Since $\xi$ can be chosen arbitrarily in $(0,T-\Delta t)$,
the above bound always holds for $\omega_{\max}(t)$ for $t\in(\Delta t,T)$.
Therefore, by taking
\begin{equation*}
  C_0:=\max\{M_1,M_2\},
\end{equation*}
we have
\begin{equation*}
  \omega(\theta,t)\leq C_0,
\quad
(\theta,t)\in S^1\times[0,T).
\end{equation*}
This completes the proof.
\end{proof}


\subsection{Lower bound and higher-order estimates on curvature}

We next derive estimates for the derivatives of $\omega$.
For convenience, set
\begin{equation*}
  \omega_n=\frac{\partial^n\omega}{\partial\theta^n},
\quad
n=1,2,3,\ldots.
\end{equation*}

\begin{lemma}\label{7}
Assume that $\gamma(t)$ is a smooth convex solution of the flow
\eqref{1.7} on $[0,T)$. Then there exists a constant
$C_1>0$ such that
\begin{equation*}
  |\omega_\theta(\theta,t)|\leq C_1,
\quad
(\theta,t)\in S^1\times[0,T).
\end{equation*}
\end{lemma}

\begin{proof}
Let $F=\omega_1+\beta\omega$ with $\beta>0$ to be chosen later. Then
\begin{equation*}
F_\theta=\omega_2+\beta\omega_1,
\quad
F_{\theta\theta}=\omega_3+\beta\omega_2.
\end{equation*}
The evolution of $\omega_1$ is given by
\begin{align*}
(\omega_1)_t
&=\frac{\alpha m\omega^{1+\frac1\alpha}}{\varphi}\omega_3
+\left[\alpha m\left(\frac{\omega^{1+\frac{1}{\alpha}}}{\varphi}\right)_{\theta}
+\frac{3\alpha m_{\theta}\omega^{1+\frac{1}{\alpha}}}{\varphi}\right]\omega_2                                      \\
&\quad
+\left[\frac{\alpha\omega^{1+\frac{1}{\alpha}}}{\varphi}(m+3m_{\theta\theta})
+2\alpha m_{\theta}\left(\frac{\omega^{1+\frac{1}{\alpha}}}{\varphi}\right)_{\theta}
+\frac{(\alpha+1)(m+m_{\theta\theta})}{\varphi}\omega^{\frac{1}{\alpha}}
(\omega-\lambda(t))\right]\omega_1
\\
&\quad
+\alpha\left(\frac{m+m_{\theta\theta}}{\varphi}\right)_{\theta}\omega^{1+\frac{1}{\alpha}}
(\omega-\lambda(t)).
\end{align*}
A direct computation shows that
\begin{equation*}
  F_t
=\frac{\alpha m\omega^{1+\frac{1}{\alpha}}}{\varphi}F_{\theta\theta}
+\left(\alpha m\left(\frac{\omega^{1+\frac{1}{\alpha}}}{\varphi}\right)_{\theta}
+\frac{3\alpha m_{\theta}\omega^{1+\frac{1}{\alpha}}}{\varphi}\right)F_{\theta}
+G(\theta,t),
\end{equation*}
where
\begin{equation*}
\begin{aligned}
G(\theta,t)
=&
\omega^{\frac{1}{\alpha}}\Bigg[\frac{(\alpha+1)(2m_{\theta-}\beta m)}{\varphi}\omega_1^2
\\
&\quad
+\Bigg(\left(\frac{2\alpha m_{\theta\theta}-\alpha\beta m_{\theta}}{\varphi}
+\frac{\alpha\beta m\varphi_{\theta}-2\alpha m_{\theta}\varphi_{\theta}}{\varphi^2}
+\frac{\alpha(m_{\theta\theta}+m)}{\varphi}\right)\omega
\\
&\qquad\quad
+\frac{(\alpha+1)(m_{\theta\theta}+m)}{\varphi}\bigl(\omega-\lambda(t)\bigr)\Bigg)\omega_1
\\
&\quad
+\alpha\Bigg(\frac{m_{\theta\theta\theta}+m_\theta+\beta(m_{\theta\theta}+m)}{\varphi}
-\frac{(m_{\theta\theta}+m)\varphi_{\theta}}{\varphi^2}\Bigg)
\omega\bigl(\omega-\lambda(t)\bigr)\Bigg].
\end{aligned}
\end{equation*}
Note that $m$, $\varphi$ are smooth positive functions on $S^1$.
We may choose $\beta>0$ sufficiently large such that $\frac{(\alpha+1)(2m_\theta-\beta m)}{\varphi}\leq -A$ for some constant $A>0$ independent of time.
Fix $t_1\in(0,T)$, and suppose that $F(\theta^*,t^*)=\max_{S^1\times[0,t_1]}F(\theta,t)$
for some $(\theta^*,t^*)\in S^1\times[0,t_1]$. We may assume that $t^*>0$.
By \eqref{3.6}, both $\omega$ and $\lambda(t)$ are uniformly bounded in time.
Hence, there exist positive constants $B$ and $C$, independent of time, such that
\begin{equation*}
G(\theta,t)\leq\left(-A\omega_1^2+B|\omega_1|+C\right)\omega^{\frac1\alpha}.
\end{equation*}
At $(\theta^*,t^*)$, we have
\begin{equation*}
F_t\geq0,\quad
F_\theta=0,\quad
F_{\theta\theta}\leq0.
\end{equation*}
Since $\omega$ is uniformly bounded, sufficiently large $F(\theta^*,t^*)$ implies that $\omega_1(\theta^*,t^*)$ is sufficiently large and positive.
Thus, $G(\theta^*,t^*)<0$, and the evolution equation of $F$ gives $F_t(\theta^*,t^*)<0$, a contradiction.
By the maximum principle, $F$ has a time-independent upper bound. Since $\omega$ is uniformly bounded, $\omega_1$ also has a time-independent upper bound.

To derive the time-independent lower bound for $\omega_1$, we set $\widetilde{F}=-\omega_1+\beta\omega$ for a suitable constant $\beta>0$ and apply the same argument as above.
\end{proof}

\begin{lemma}\label{8}
Under the flow \eqref{1.7} with $0<\alpha\leq1$, there exists a constant $c_0>0$ independent of time such that
\begin{equation*}
  \kappa(\theta,t)\geq c_0>0,
  \quad
  (\theta,t)\in S^1\times[0,T).
\end{equation*}
\end{lemma}

\begin{proof}
We first consider the case $0<\alpha<1$.
Since $\omega=\varphi^{\alpha}\kappa^{\alpha}$, the estimates in Lemmas~\ref{6} and~\ref{7} imply $\left| \kappa^{-1}(\kappa^{\alpha})_s \right| \leq c_1$ for some constant $c_1>0$ independent of time, where $s$ is the arc-length parameter of $\gamma(\cdot,t)$.
Therefore, $\left|(\kappa^{\alpha-1})_s\right|\leq c_2$ for some constant $c_2>0$ independent of time.
For any $s_1\leq s_2\in[0,L(t))$,  we have
\begin{equation}\label{3.7}
  \left|\kappa^{\alpha-1}(s_2,t)-\kappa^{\alpha-1}(s_1,t)\right|
  \leq
  \int_{s_1}^{s_2}\left|(\kappa^{\alpha-1})_s\right|ds\leq c_2L(t)
  \leq c_3,
\end{equation}
where the last inequality follows from $L_{\sigma}(t)=L_{\sigma}(0)$,
$L(t)\leq \frac{L_\sigma(t)}{\min_{S^1}\sigma}$ and $c_3$ is a constant independent of time.
Notice that for any $t\in[0,T)$, it holds that
\begin{equation*}
  L_{\sigma}(0)\geq L_{\sigma}(t)
  =\int_0^{2\pi}\frac{\sigma}{\kappa(\theta,t)}\,d\theta
  \geq
  \frac{2\pi\min_{S^1}\sigma}{\kappa(\theta_1(t),t)},
\end{equation*}
where
$\kappa(\theta_1(t),t)=\max_{S^1}\kappa(\theta,t)$.
Assume that
$\kappa(\theta_2(t),t)=\min_{S^1}\kappa(\theta,t)$.
By applying \eqref{3.7} at the points corresponding to $\theta_1(t)$ and $\theta_2(t)$, we obtain
\begin{equation*}
  \kappa^{\alpha-1}(\theta_2(t),t)
  \leq
  c_3+\left(\frac{2\pi\min_{S^1}\sigma}{L_{\sigma}(0)}\right)^{\alpha-1}
  :=c_4,
\end{equation*}
i.e.,
\begin{equation*}
  \kappa(\theta_2(t),t)
  \geq
  c_4^{-\frac{1}{1-\alpha}}>0.
\end{equation*}
Thus, the lower bound for the curvature has been obtained for the case $0<\alpha<1$.
For the case $\alpha=1$, we can use $\ln\kappa$ instead of $\kappa^{\alpha-1}$ to obtain the desired result.
\end{proof}


\begin{lemma}\label{9}
For the flow \eqref{1.7}, if we assume $\omega$ and $\omega_\theta$ satisfy
\begin{equation*}
  0<c_0\leq \omega\leq C_0,
  \quad
  |\omega_{\theta}|\leq C_1,
  \quad
  (\theta,t)\in S^1\times[0,T),
\end{equation*}
then there exists a positive time-independent constant $C_2$ such that
\begin{equation*}
  |\omega_{\theta\theta}(\theta,t)|\leq C_2,
  \quad
  (\theta,t)\in S^1\times[0,T).
\end{equation*}
\end{lemma}

\begin{proof}
For convenience, we rewrite \eqref{2.5} as
\begin{equation*}
  \frac{\partial\omega}{\partial t}
=A(\omega)\omega_2+B(\omega)\omega_1+C(\omega),
\end{equation*}
where  $A(\omega)=\frac{\alpha m\omega^{1+\frac{1}{\alpha}}}{\varphi},
B(\omega)=\frac{2\alpha m_{\theta}\omega^{1+\frac{1}{\alpha}}}{\varphi},
C(\omega)=\frac{\alpha(m+m_{\theta\theta})}{\varphi}
\omega^{1+\frac{1}{\alpha}}\bigl(\omega-\lambda(t)\bigr)$.
Then we have
\begin{equation*}
(\omega_1)_t=A_1(\omega)\omega_3+B_1(\omega,\omega_1)\omega_2+\widetilde{C}_1(\omega,\omega_1),
\end{equation*}
where $A_1(\omega)=A(\omega), B_1(\omega,\omega_1)=\frac{\partial A}
{\partial\theta}+B(\omega)$
and  $\widetilde{C}_1(\omega,\omega_1)
=\frac{\partial B(\omega)}{\partial\theta}\omega_1
+\frac{\partial C(\omega)}{\partial\theta}$.
 Furthermore, we have
\begin{equation*}
  (\omega_2)_t
=A_2(\omega)\omega_4+B_2(\omega,\omega_1)\omega_3+C_{21}(\omega,\omega_1)\omega_2^2
+C_{22}(\omega,\omega_1)\omega_2+C_{23}(\omega,\omega_1),
\end{equation*}
where
$A_2(\omega)
=A_1(\omega),B_2(\omega,\omega_1)=\frac{\partial A_1}{\partial\theta}+B_1(\omega,\omega_1)$  and   $C_{21}(\omega,\omega_1)\omega_2^2+C_{22}(\omega,\omega_1)\omega_2
+C_{23}(\omega,\omega_1)=\frac{\partial B_1}{\partial\theta}\omega_2
+\frac{\partial\widetilde{C}_1}{\partial\theta}$.

Set $F=\omega_2+\beta\omega_1^2$, where $\beta>0$ is a constant to be chosen later.
A direct computation shows that
\begin{equation*}
  F_t=A_2F_{\theta\theta}+B_2F_\theta+G(\theta,t),
\end{equation*}
where $G(\theta,t)=(-2\beta A_2+C_{21})\omega_2^2+(-2\beta B_2\omega_1+C_{22}+2\beta B_1\omega_1)\omega_2+2\beta \widetilde{C}_1\omega_1+C_{23}$.
Since  $A_2=A =\frac{\alpha m}{\varphi}\omega^{1+\frac{1}{\alpha}}$,
$C_{21}=\frac{\partial A}{\partial\omega}
=\frac{(\alpha+1)m}{\varphi}\omega^{\frac{1}{\alpha}}$,
we have $-2\beta A_2+C_{21}=\frac{m}{\varphi}\omega^{\frac{1}{\alpha}}
\left(\alpha+1-2\alpha\beta\omega\right)$.
Choose $\beta\geq\frac{\alpha+2}{2\alpha c_0}$.
Since $\omega\geq c_0$, it follows that $ -2\beta A_2+C_{21}\leq-\frac{\min_{S^1}m}
{\max_{S^1}\varphi}c_0^{\frac{1}{\alpha}}:=-\delta<0$.
Moreover, from the assumptions and $0<\lambda(t)\leq C_0$, it follows that all the remaining coefficients in $G$ are uniformly bounded.
Hence, there exists a constant $K>0$ such that
$G(\theta,t)\leq-\delta\omega_2^2+K|\omega_2|+K$.

Suppose that $F$ attains its maximum on $S^1\times[0,t_1]$ at an interior point
$(\theta_*,t_*)$.
At this point,
\begin{equation*}
  F_t\geq0,\quad
    F_\theta=0,\quad
    F_{\theta\theta}\leq0.
\end{equation*}
Since $|\omega_1|\leq C_1$, sufficiently large $F$ implies that $\omega_2=F-\beta\omega_1^2$ is sufficiently large and positive.
Therefore, the above implies $G(\theta_*,t_*)<0$, contradicting the evolution equation of $F$.
Therefore, $F$ is bounded above.
It follows that $\omega_2$ is uniformly bounded from above.

Applying the same argument to $\widetilde F=-\omega_2+\beta\omega_1^2$
gives a uniform lower bound for $\omega_2$.
Consequently, $|\omega_{\theta\theta}(\theta,t)|\leq C_2$ on $S^1\times[0,T)$.
\end{proof}


\begin{lemma}\label{10}
Consider the  ALP flow \eqref{1.7}.
For any  integer $k>2$, if $\omega,|\omega_1|,|\omega_2|,\ldots,|\omega_{k-1}|$ are
uniformly bounded from above and $\omega$ has a uniform positive lower bound
on $S^1\times[0,T)$, then $|\omega_k|$ is uniformly bounded from above
on $S^1\times[0,T)$.
\end{lemma}

\begin{proof}
Rewriting \eqref{2.5}, we have
\begin{equation}\label{3.8}
  \omega_t=A\omega_2+B\omega_1+C.
\end{equation}
Here $A=\frac{\alpha m}{\varphi}\omega^{1+\frac{1}{\alpha}}$,
$B=\frac{2\alpha m_{\theta}}{\varphi}\omega^{1+\frac{1}{\alpha}}$, and
$C=\frac{\alpha(m_{\theta\theta}+m)}{\varphi}
\omega^{1+\frac{1}{\alpha}}\bigl(\omega-\lambda(t)\bigr)$.
Let $D_{\theta}$ denote the total derivative with respect to $\theta$ and let the subscript $\omega$ denote the partial derivative with respect to $\omega$.

Differentiating \eqref{3.8} $k$ and $k-1$ times, respectively, we obtain
\begin{equation}\label{3.9}
  (\omega_k)_t=A\omega_{k+2}+B_k\omega_{k+1}+\mathcal R_k,
\end{equation}
\begin{equation}\label{3.10}
  (\omega_{k-1})_t=A\omega_{k+1}+B_{k-1}\omega_k+\mathcal R_{k-1},
\end{equation}
where $B_j=B+jD_\theta A$, $j=k-1,k$.
The term $\mathcal R_j$ depends only on $\theta$, $t$, $\omega$, $\omega_1,\ldots,\omega_j$.
For $k>2$, the Leibniz rule shows that $\mathcal R_k$ depends at most linearly on $\omega_k$.
The assumed bounds and the uniform bound for $\lambda(t)$ therefore imply that there exists a time-independent constant $C>0$ such that
\begin{equation}\label{3.11}
  |\mathcal R_k|\leq C(1+|\omega_k|),
  \quad
  |\mathcal R_{k-1}|\leq C.
\end{equation}
It also follows from the definition of $B_j$ that $B_{k-1}-B_k=-D_{\theta} A$.

Let $F=\omega_k+\beta\omega_{k-1}^2$, where $\beta>0$ is fixed.
Then
\begin{equation*}
  F_{\theta}=\omega_{k+1}+2\beta\omega_{k-1}\omega_k,
\end{equation*}
and
\begin{equation*}
  F_{\theta\theta}=\omega_{k+2}+2\beta\omega_k^2+2\beta\omega_{k-1}\omega_{k+1}.
\end{equation*}
Using \eqref{3.9} and \eqref{3.10}, we obtain
\begin{equation*}
  F_t =A\omega_{k+2}+B_k\omega_{k+1}+\mathcal R_k+2\beta\omega_{k-1}
  \left(A\omega_{k+1}+B_{k-1}\omega_k+\mathcal R_{k-1}\right).
\end{equation*}
On the other hand,
\begin{equation*}
  AF_{\theta\theta}+B_kF_{\theta}
  =A\omega_{k+2}+B_k\omega_{k+1}+2\beta A\omega_k^2
  +2\beta A\omega_{k-1}\omega_{k+1}+2\beta B_k\omega_{k-1}\omega_k.
\end{equation*}
Consequently,
\begin{equation}\label{3.12}
  \begin{aligned}
  F_t
  =&
  AF_{\theta\theta}+B_kF_{\theta}-2\beta A\omega_k^2+2\beta(B_{k-1}-B_k)\omega_{k-1}\omega_k
  \\
  &\quad
  +\mathcal R_k+2\beta\omega_{k-1}\mathcal R_{k-1}
  \\
  =&
  AF_{\theta\theta}+B_kF_{\theta}-2\beta A\omega_k^2-2\beta D_\theta A\,\omega_{k-1}\omega_k
  \\
  &\quad
  +\mathcal R_k+2\beta\omega_{k-1}\mathcal R_{k-1}.
  \end{aligned}
\end{equation}

By the assumptions of the lemma, both $\omega_{k-1}$ and $D_\theta A$ are uniformly bounded. Moreover, since $\omega$ has a uniform positive lower bound, there exists a constant $a_0>0$ such that  $A\geq a_0$  on $S^1\times[0,T)$.
It follows from \eqref{3.11} and \eqref{3.12} that
\begin{equation}\label{3.13}
  F_t\leq AF_{\theta\theta}+B_kF_{\theta}-2\beta a_0\omega_k^2+C|\omega_k|+C.
\end{equation}

Fix $t_1\in(0,T)$, and suppose that $F$ attains its maximum on $S^1\times[0,t_1]$ at a point $(\theta_*,t_*)$ with $t_*>0$.
At this point,
\begin{equation*}
  F_t\geq0,
  \quad
  F_\theta=0,
  \quad
  F_{\theta\theta}\leq0.
\end{equation*}
Since $\omega_{k-1}$ is uniformly bounded, sufficiently large $F(\theta_*,t_*)$ implies that
$\omega_k(\theta_*,t_*)$ is sufficiently large and positive.
It then follows from \eqref{3.13} that
\begin{equation*}
  0
  \leq
  -2\beta a_0\omega_k^2
  +C|\omega_k|
  +C,
\end{equation*}
which is impossible when $\omega_k$ is sufficiently large.
Thus, $F$ has a time-independent upper bound, and hence $\omega_k$ is uniformly bounded from above.

Similarly, applying the same argument to $\widetilde F=-\omega_k+\beta\omega_{k-1}^2$
gives a uniform lower bound for $\omega_k$. Consequently,
\begin{equation*}
  |\omega_k|\leq C_k,
  \quad
  (\theta,t)\in S^1\times[0,T).
\end{equation*}

\end{proof}

Now we can state the global existence of the flow \eqref{1.7}.

\begin{lemma}\label{11}
Assume that the initial curve $\gamma_0$ is a smooth convex closed curve.
The flow \eqref{1.7} has a smooth convex solution and the solution exists on the time
interval $[0,\infty)$.
\end{lemma}

\begin{proof}
Let $[0,T_{\max})$ be the maximal existence interval of the solution.
Suppose, for contradiction, that $T_{\max}<\infty$. To obtain the a priori derivative estimates required for continuation, we apply Lemmas~\ref{7}, \ref{9} and~\ref{10}.

For $0<\alpha\leq1$, Lemma~\ref{8}, together with $\omega=(\varphi\kappa)^{\alpha}$ and
$\min_{S^1}\varphi>0$, provides the required positive lower bound for $\omega$.

For $\alpha>1$, Lemma~\ref{6} and \eqref{2.3} imply that $0<\omega(\theta,t)\leq C_0$ and $0<\lambda(t)\leq C_0$.
Hence, $\bigl|\frac{\alpha(m+m_{\theta\theta})}{\varphi}\omega^{\frac{1}{\alpha}}
(\omega-\lambda(t))\bigr|\leq B_0$ for $(\theta,t)\in S^1\times[0,T_{\max})$.
Then, we have
\begin{equation*}
\omega(\theta,t)
\geq e^{-(B_0+1)t}\min_{S^1}\omega(\cdot,0)
\geq e^{-(B_0+1)T_{\max}}\min_{S^1}\omega(\cdot,0)>0,
\quad
(\theta,t)\in S^1\times[0,T_{\max}).
\end{equation*}

Thus, for every $\alpha>0$, $\omega$ has a positive lower bound on $S^1\times[0,T_{\max})$. Together with Lemmas~\ref{6} and~\ref{7}, this verifies the assumptions of Lemma~\ref{9}. Lemmas~\ref{9} and~\ref{10} then yield the derivative estimates required for continuation. Consequently, by the continuation criterion in Lemma~\ref{4}, the solution can be extended beyond $T_{\max}$, contradicting the definition of the maximal existence time. Therefore, $T_{\max}=\infty$, and the solution exists for all time.

\end{proof}


\section{The convergence of the anisotropic curvature}

The estimates on $\omega$ and $\omega_\theta$ in the last section
imply the compactness of the anisotropic curvature function.
We next prove the convergence of the anisotropic curvature via the
Lyapunov functional method.
Furthermore, the convergence result gives a time-independent positive lower bound for the anisotropic curvature, which yields time-independent bounds for its higher-order derivatives.

\begin{lemma}\label{12}
Under the flow \eqref{1.7}, we have
\begin{equation}\label{4.1}
\frac{dA(t)}{dt}\to0,
\quad
t\to\infty.
\end{equation}
\end{lemma}

\begin{proof}
By \eqref{2.10}, the evolution of the enclosed area satisfies
\begin{equation*}
  \frac{dA}{dt}
  =-\int_0^{2\pi}m\varphi\omega^{1-\frac{1}{\alpha}}d\theta
  +\frac{\int_0^{2\pi}m\varphi\omega d\theta}{\int_0^{2\pi}m\varphi d\theta}
  \int_0^{2\pi}m\varphi\omega^{-\frac{1}{\alpha}} d\theta
  \geq0.
\end{equation*}
Thus by \eqref{2.5} we obtain
\begin{align}
\frac{d^2A}{dt^2}
&=-(\alpha-1)\int_0^{2\pi}m\omega
\left(m\omega_{\theta\theta}+2m_{\theta}\omega_{\theta}
+(m+m_{\theta\theta})(\omega-\lambda)\right)d\theta
\notag\\
&\quad
+\frac{\alpha\int_0^{2\pi}m\varphi\omega^{-\frac{1}{\alpha}}d\theta}
{\int_0^{2\pi}m\varphi d\theta}\int_0^{2\pi}m\omega^{1+\frac{1}{\alpha}}
\left(m\omega_{\theta\theta}+2m_{\theta}\omega_{\theta}
+(m+m_{\theta\theta})(\omega-\lambda)\right)d\theta
\notag\\
&\quad
-\frac{\int_0^{2\pi}m\varphi\omega d\theta}
{\int_0^{2\pi}m\varphi d\theta}\int_0^{2\pi}m\left(m\omega_{\theta\theta}
+2m_{\theta}\omega_{\theta}+(m+m_{\theta\theta})(\omega-\lambda)\right)d\theta
\notag\\
\label{4.2}
&=
(\alpha-1)\int_0^{2\pi}m^2\omega_{\theta}^2d\theta
-(\alpha-1)\int_0^{2\pi}m\omega(m+m_{\theta\theta})(\omega-\lambda)d\theta
\notag\\
&\quad
-(\alpha+1)\frac{\int_0^{2\pi}m\varphi\omega^{-\frac{1}{\alpha}}d\theta}
{\int_0^{2\pi}m\varphi d\theta}
\int_0^{2\pi}m^2\omega^{\frac{1}{\alpha}}\omega_{\theta}^2d\theta
\notag\\
&\quad
+\alpha\frac{\int_0^{2\pi}m\varphi\omega^{-\frac{1}{\alpha}}d\theta}
{\int_0^{2\pi}m\varphi d\theta}
\int_0^{2\pi}m\omega^{1+\frac{1}{\alpha}}(m+m_{\theta\theta})(\omega-\lambda)d\theta
\notag\\
&\quad
-\frac{\int_0^{2\pi}m\varphi\omega d\theta}{\int_0^{2\pi}m\varphi d\theta}
\int_0^{2\pi}m(m+m_{\theta\theta})(\omega-\lambda)d\theta .
\end{align}
Lemma~\ref{6}, Lemma~\ref{7} and \eqref{2.3} yield uniform bounds for $\omega$, $\omega_{\theta}$ and $\lambda(t)$, respectively.
Combining with the inequality
\begin{equation*}
  L_{\sigma}^2(0)=L_{\sigma}^2(t)\geq4|W_{\sigma}|A(t)\geq 4|W_{\sigma}|A(0),
\end{equation*}
and
\begin{equation*}
\int_0^{2\pi}m\varphi\omega^{-\frac{1}{\alpha}}d\theta
=\int_{\gamma}mds
\leq\left(\max_{S^1}\frac{m}{\sigma}\right)L_{\sigma}(0),
\end{equation*}
we know there exists a time-independent constant $C$ such that
\begin{equation}\label{4.3}
\left|\frac{d^2A(t)}{dt^2}\right|\leq C,
\quad
\forall t\in[0,\infty).
\end{equation}
Note that
\begin{equation*}
  \int_0^{\infty}\frac{dA}{d\tau}d\tau=A(\infty)-A(0)<\infty.
\end{equation*}
Using \eqref{4.3}, we can deduce that
\begin{equation*}
  \frac{dA}{dt}(t)\rightarrow 0,
  \quad
  t\rightarrow \infty.
  \qedhere
\end{equation*}
\end{proof}

We recall the following inequality due to Andrews~\cite{1}.

\begin{lemma}\label{13}
Let $M$ be a compact Riemannian manifold with volume form $d\mu$, and let $\xi$ be a continuous function on $M$.
Then, for any decreasing continuous function $F:\mathbb R\to\mathbb R$, one has
\begin{equation}\label{4.5}
\int_M\xi d\mu\int_MF(\xi)d\mu
\geq
\int_Md\mu\int_M\xi F(\xi)d\mu.
\end{equation}
If $F$ is strictly decreasing, then equality holds if and only if $\xi$ is constant on $M$.
\end{lemma}


\begin{lemma}\label{14}
For the flow \eqref{1.7}, we have
\begin{equation}\label{4.6}
\lim_{t\to\infty}\left\|\kappa_{\sigma}(\cdot,t)
-\frac{2|W_{\sigma}|}{L_{\sigma}(0)}\right\|_{C^k(S^1)}=0,
\end{equation}
for any $k\in N$.
\end{lemma}

\begin{proof}
We first prove the $C^0$ convergence.
Let $\{t_i\}_{i=1}^{\infty}$ be an arbitrary sequence tending to infinity.
By the uniform bounds for $\omega$ and $\omega_{\theta}$ and the Arzel\`a--Ascoli theorem, there exists a subsequence, still denoted by $\{t_i\}_{i=1}^{\infty}$,
such that $\kappa_{\sigma}(\theta,t_i)\longrightarrow\kappa_{\sigma}^*(\theta)$
in $C^0(S^1)$ as $i\to\infty$.
When $\alpha>1$, we also have  $\kappa_{\sigma}^{\alpha-1}(\cdot,t_i)
\longrightarrow(\kappa_{\sigma}^*)^{\alpha-1}$  in $C^0(S^1)$.
For $0<\alpha\leq1$, the same conclusion follows from the positive lower bound in Lemma~\ref{8}.

Since the anisotropic length is preserved, we have
\begin{equation*}
\int_0^{2\pi}\frac{m\varphi}{\kappa_{\sigma}(\theta,t_i)}d\theta
=\int_{\gamma(t_i)}mds
\leq\max_{S^1}\frac{m}{\sigma}\int_{\gamma(t_i)}\sigma ds
=\max_{S^1}\frac{m}{\sigma}L_{\sigma}(0).
\end{equation*}
Since these integrals are uniformly bounded, we may choose a further subsequence, still denoted by $\{t_i\}$, along which $\int_0^{2\pi}\frac{m\varphi}{\kappa_{\sigma}(\theta,t_i)}d\theta$ converges as $i\to\infty$.
By \eqref{4.1}, we obtain
\begin{equation}\label{4.7}
0=\lim_{i\to\infty}\frac{dA}{dt}(t_i)
=-\int_0^{2\pi}m\varphi(\kappa_{\sigma}^*)^{\alpha-1}d\theta
+\frac{\int_0^{2\pi}
m\varphi(\kappa_{\sigma}^*)^{\alpha}d\theta}{\int_0^{2\pi}m\varphi d\theta}
\lim_{i\to\infty}\int_0^{2\pi}\frac{m\varphi}{\kappa_{\sigma}(\theta,t_i)}d\theta.
\end{equation}

By Fatou's lemma, we have
\begin{equation*}
\int_0^{2\pi}\frac{m\varphi}{\kappa_{\sigma}^*}d\theta
\leq
\lim_{i\to\infty}\int_0^{2\pi}\frac{m\varphi}{\kappa_{\sigma}(\theta,t_i)}d\theta
<\infty.
\end{equation*}
Thus, $\kappa_\sigma^*>0$ almost everywhere on $S^1$,
and it follows from \eqref{4.7} that
\begin{equation}\label{4.8}
\int_0^{2\pi}\frac{m\varphi}{\kappa_{\sigma}^*}d\theta
\int_0^{2\pi}m\varphi(\kappa_{\sigma}^*)^{\alpha}d\theta
\leq
\int_0^{2\pi}m\varphi d\theta
\int_0^{2\pi}m\varphi(\kappa_{\sigma}^*)^{\alpha-1}d\theta.
\end{equation}

On the other hand, set
\begin{equation*}
F(\xi)=\frac{1}{\xi^{\frac{1}{\alpha}}+\varepsilon},
\quad
\xi\geq0,
\quad
\varepsilon>0.
\end{equation*}
Applying inequality \eqref{4.5}, we have
\begin{equation*}
\int_0^{2\pi}m\varphi(\kappa_{\sigma}^*)^{\alpha}d\theta
\int_0^{2\pi}\frac{m\varphi}{\kappa_{\sigma}^*+\varepsilon}d\theta
\geq
\int_0^{2\pi}m\varphi d\theta
\int_0^{2\pi}\frac{m\varphi(\kappa_{\sigma}^*)^\alpha}
{\kappa_{\sigma}^*+\varepsilon}d\theta.
\end{equation*}
Letting $\varepsilon\to0^+$, the dominated convergence
theorem gives
\begin{equation}\label{4.9}
\int_0^{2\pi}m\varphi(\kappa_{\sigma}^*)^{\alpha}d\theta
\int_0^{2\pi}\frac{m\varphi}{\kappa_{\sigma}^*}d\theta
\geq
\int_0^{2\pi}m\varphi d\theta
\int_0^{2\pi}m\varphi(\kappa_\sigma^*)^{\alpha-1}d\theta.
\end{equation}

Combining \eqref{4.8} and \eqref{4.9}, we obtain
\begin{equation*}
\int_0^{2\pi}m\varphi(\kappa_{\sigma}^*)^{\alpha}d\theta
\int_0^{2\pi}\frac{m\varphi}{\kappa_{\sigma}^*}d\theta
=
\int_0^{2\pi}m\varphi d\theta
\int_0^{2\pi}m\varphi(\kappa_{\sigma}^*)^{\alpha-1}d\theta.
\end{equation*}
Equivalently,
\begin{equation}\label{4.10}
\int_0^{2\pi}\int_0^{2\pi}m(x)m(y)\varphi(x)\varphi(y)
\left(\frac{1}{\kappa_{\sigma}^*(x)}-\frac{1}{\kappa_{\sigma}^*(y)}\right)
((\kappa_{\sigma}^*(x))^{\alpha}-(\kappa_{\sigma}^*(y))^{\alpha})
dxdy=0.
\end{equation}
Since the integrand in \eqref{4.10} is nonpositive almost everywhere and $\kappa_{\sigma}^*>0$ almost everywhere, we conclude that $\kappa_{\sigma}^*$ is constant almost everywhere.
By continuity, $\kappa_{\sigma}^*\equiv C^*$ on $S^1$ for some constant $C^*>0$.

Using the preservation of the anisotropic length, we have
\begin{equation*}
L_{\sigma}(0)
=\lim_{i\to\infty}\int_0^{2\pi}\frac{\sigma\varphi}{\kappa_{\sigma}(\theta,t_i)}d\theta
=\frac{1}{C^*}\int_0^{2\pi}\sigma\varphi d\theta
=\frac{2|W_{\sigma}|}{C^*}.
\end{equation*}
Hence,
\begin{equation*}
C^*=\frac{2|W_{\sigma}|}{L_\sigma(0)}.
\end{equation*}
Since the original sequence was arbitrary and the limit does not depend on the choice of subsequence, we obtain
\begin{equation*}
\kappa_\sigma(\cdot,t)\longrightarrow\frac{2|W_{\sigma}|}{L_{\sigma}(0)}
\quad\text{in }C^0(S^1)
\quad\text{as }t\to\infty.
\end{equation*}

After the above convergence is established, the anisotropic curvature has a time-independent positive lower bound.
Lemmas~\ref{7}, \ref{9} and~\ref{10} then give time-independent bounds for all higher-order spatial derivatives of $\omega=\kappa_{\sigma}^{\alpha}$.
Together with the two-sided bounds for $\omega$, this also gives time-independent bounds for all higher-order spatial derivatives of $\kappa_{\sigma}$.
Thus, the above $C^0$ convergence improves to $C^k$ convergence for every positive integer $k$.
\end{proof}


\begin{thebibliography}{99}

\bibitem{1}
B.~Andrews,
\emph{Evolving convex curves},
Calc. Var. Partial Differential Equations
\textbf{7} (1998), 315--371.

\bibitem{2}
B.~Andrews, Y.~Lei, Y.~Wei, and C.~Xiong,
\emph{Volume preserving flows in anisotropic geometries},
Calc. Var. Partial Differential Equations
\textbf{64} (2025), no.~9, Paper No.~287, 49~pp.

\bibitem{3}
S.~Angenent and M.~E.~Gurtin,
\emph{Multiphase thermomechanics with interfacial structure,
II: Evolution of an isothermal interface},
Arch. Ration. Mech. Anal.
\textbf{108} (1989), 323--391.

\bibitem{4}
S.~Angenent and M.~E.~Gurtin,
\emph{Anisotropic motion of a phase interface:
Well-posedness of the initial value problem and qualitative
properties of the interface},
J. Reine Angew. Math.
\textbf{446} (1994), 1--47.

\bibitem{5}
M.~Bene\v{s}, S.~Yazaki, and M.~Kimura,
\emph{Computational studies of non-local anisotropic
Allen--Cahn equation},
Math. Bohem.
\textbf{136} (2011), 429--437.

\bibitem{6}
K.S.~Chou and X.P.~Zhu,
\emph{The Curve Shortening Problem},
Chapman \& Hall/CRC,  2001.

\bibitem{7}
M.~E.~Gage and R.~S.~Hamilton,
\emph{The heat equation shrinking convex plane curves},
J. Differential Geom.
\textbf{23} (1986), 69--96.

\bibitem{8}
M.~E.~Gage,
\emph{On an area-preserving evolution equation for plane curves},
in \emph{Nonlinear Problems in Geometry}
(Mobile, Ala., 1985),
Contemp. Math., vol.~51,
Amer. Math. Soc., Providence, RI, 1986,
pp.~51--62.

\bibitem{9}
M.~E.~Gage,
\emph{Evolving plane curves by curvature in relative geometries},
Duke Math. J.
\textbf{72} (1993), 441--466.

\bibitem{10}
M.~E.~Gage and Y.~Li,
\emph{Evolving plane curves by curvature in relative geometries II},
Duke Math. J.
\textbf{75} (1994), 79--98.

\bibitem{11}
L.Y.~Gao and Y.T.~Zhang,
\emph{On Yau's problem of evolving one curve to another:
Convex case},
J. Differential Equations
\textbf{266} (2019), 179--201.

\bibitem{14}
Y.~Giga and N.~Po\v{z}\'ar,
\emph{Motion by crystalline-like mean curvature: a survey},
Bull. Math. Sci.
\textbf{12} (2022), Paper No.~2230004, 68~pp.

\bibitem{15}
M.~A.~Grayson,
\emph{The heat equation shrinks embedded plane curves to
round points},
J. Differential Geom.
\textbf{26} (1987), 285--314.

\bibitem{16}
M.~E.~Gurtin,
\emph{Thermomechanics of Evolving Phase Boundaries in the Plane},
Oxford Mathematical Monographs,
Clarendon Press, Oxford University Press, New York, 1993.

\bibitem{17}
M.~E.~Gurtin,
\emph{Multiphase thermomechanics with interfacial structure,
I: Heat conduction and the capillary balance law},
Arch. Ration. Mech. Anal.
\textbf{104} (1988), 195--221.

\bibitem{18}
M.~E.~Gurtin,
\emph{Toward a nonequilibrium thermomechanics of
two-phase materials},
Arch. Ration. Mech. Anal.
\textbf{100} (1988), 275--312.

\bibitem{19}
L.S.~Jiang and S.L.~Pan,
\emph{On a non-local curve evolution problem in the plane},
Comm. Anal. Geom.
\textbf{16} (2008), 1--26.

\bibitem{20}
L.~Liu, D.H.~Tsai, and X.L.~Wang,
\emph{On length-preserving and area-preserving anisotropic
curvature flow of convex closed plane curves},
J. Evol. Equ.
\textbf{25} (2025), Paper No.~2, 24~pp.

\bibitem{21}
L.~Ma and L.~Cheng,
\emph{A non-local area-preserving curve flow},
Geom. Dedicata
\textbf{171} (2014), 231--247.

\bibitem{22}
L.~Ma and A.-Q.~Zhu,
\emph{On a length-preserving curve flow},
Monatsh. Math.
\textbf{165} (2012), 57--78.

\bibitem{23}
W.~W.~Mullins,
\emph{Two-dimensional motion of idealized grain boundaries},
J. Appl. Phys.
\textbf{27} (1956), 900--904.

\bibitem{24}
S.L.~Pan and Y.L.~Yang,
\emph{An anisotropic area-preserving flow for convex plane curves},
J. Differential Equations
\textbf{266} (2019), 3764--3786.

\bibitem{25}
G.~Sapiro and A.~Tannenbaum,
\emph{Area and length preserving geometric invariant scale-spaces},
IEEE Trans. Pattern Anal. Mach. Intell.
\textbf{17} (1995), 67--72.

\bibitem{26}
Z.~Sun,
\emph{Deforming convex curves with constant anisotropic length},
J. Geom. Anal.
\textbf{35} (2025), Paper No.~296, 21~pp.


\bibitem{28}
D.~\v{S}ev\v{c}ovi\v{c} and S.~Yazaki,
\emph{On a gradient flow of plane curves minimizing the
anisoperimetric ratio},
IAENG Int. J. Appl. Math.
\textbf{43} (2013), 160--171.

\bibitem{29}
J.~E.~Taylor and J.~W.~Cahn,
\emph{Linking anisotropic sharp and diffuse surface motion laws
via gradient flows},
J. Statist. Phys.
\textbf{77} (1994), 183--197.

\bibitem{30}
D.H.~Tsai and X.L.~Wang,
\emph{On length-preserving and area-preserving nonlocal flow
of convex closed plane curves},
Calc. Var. Partial Differential Equations
\textbf{54} (2015), 3603--3622.

\bibitem{31}
K.S.~Tso,
\emph{Deforming a hypersurface by its Gauss--Kronecker curvature},
Comm. Pure Appl. Math.
\textbf{38} (1985), 867--882.

\bibitem{32}
X.L.~Wang,
\emph{The evolution of area-preserving and length-preserving
inverse curvature flows for immersed locally convex closed
plane curves},
J. Funct. Anal.
\textbf{284} (2023), Paper No.~109744, 25~pp.

\bibitem{33}
Y.~Wei and C.~Xiong,
\emph{A fully nonlinear locally constrained anisotropic
curvature flow},
Nonlinear Anal.
\textbf{217} (2022), Paper No.~112760, 29~pp.

\bibitem{34}
Y.~Wei and C.~Xiong,
\emph{A volume-preserving anisotropic mean curvature type flow},
Indiana Univ. Math. J.
\textbf{70} (2021), 881--905.

\end{thebibliography}
\end{document}